\documentclass[12pt,letterpaper]{amsart}
\usepackage[foot]{amsaddr}
\usepackage{geometry}
\usepackage{graphicx} 
\usepackage{array}
\usepackage{graphicx} 
\usepackage{amsmath}
\usepackage{amsfonts}
\usepackage{geometry}
\usepackage{color}
\usepackage{amsthm}
\usepackage{comment}
\usepackage[utf8]{inputenc}
\usepackage{hyperref}
\usepackage{amssymb}
\newtheorem{corollary}{Corollary}[section]
\newtheorem{theorem}{Theorem}[section]
\newtheorem{definition}{Definition}[section]
\newtheorem{lemma}{Lemma}[section]
\newtheorem{proposition}{Proposition}[section]
\newtheorem*{remark}{Remark}
\newtheorem{conjecture}{Conjecture}
\title{Guessing Strategies for Shelf-Shuffling Machines}
\author[Alexander Clay]{Alexander Clay}
\date{June 2025}
\begin{document}

\thispagestyle{empty}

\begin{abstract}
    We investigate a one-time single shelf shuffle by establishing the position matrix explicitly. In some cases, we prove a no-feedback optimal guessing strategy. A general no-feedback strategy is conjectured, and asymptotics for the expected reward are given. For the complete-feedback case, we give a guessing strategy, prove that it is optimal and unique, and find the expected reward. Our results prove a conjecture of Diaconis, Fulman, and Holmes in a special case. 
\end{abstract}
\maketitle
\section{Introduction and Main Results}
\label{Section1}
Shuffling machines play a vital role in the theory of card guessing. They have practical applications to casino gambling and manufacturing, and theoretical connections to Lie groups, dynamical systems, and the RSK algorithm. A standard shuffling machine employs ten shelves. In a shuffling machine, cards are drawn from the top of the deck. Each card is placed, uniformly at random, on a shelf. The decision to place the card on the top or bottom of the existing pile on the shelf is done randomly. The shelves are then removed and the cards are pressed together, forming a shuffled deck. This setup is standard across most casinos, so studying it is crucial to gambling strategies. Observers would question how randomized the cards are after a shuffle. Gamblers attempt to gain advantages by correctly guessing the order of cards that appear. A natural test for randomness is the expected number of correct guesses, with varying levels of feedback.

The theory of optimal guessing strategies has seen rapid development over the past few years. The setup is simple. A collection of identical labeled objects is mixed. An observer, who cannot see the labeling, models the mixing using a probability distribution. Then, the observer goes through the objects and attempts to guess which labeled object is in each place. Frequently, mathematicians use the language of card-shuffling to describe this. The labeled objects are referred to as cards, and the cards are face down. The order of the cards is known to the player before they are shuffled. The cards are then shuffled, and the player attempts to guess which card ended up in each place. Different situations occur with regards to the feedback the player receives. Diaconis and Graham in \cite{Diaconis1981} classified the feedback a player may receive and provided algorithms to analyze each type. 

The player may receive no information after each guess. This is the no-feedback case. Ciucu studied no-feedback card guessing for dovetail shuffles in \cite{Ciucu}. He computed formulas for the higher powers of the position matrix by diagonalizing it. Using analytical techniques, he proved that in a deck of $2n$ cards given $s\geq 2\log(2n)+1$ shuffles, the best strategy is to guess card $1$ for the first $n$ cards and card $2n$ for the last $n$ cards. He also showed that the $2\log(2n)+1$ bound is asymptotically sharp with regards to the strategy above.

The player may instead receive partial feedback after each guess, such as a yes/no answer. In many cases, including riffle shuffles, problems in this area are open. Recent work by Diaconis et. al. in \cite{CardGuessingWithPartialFeedback} and \cite{GuessingAboutGuessing} studies the case when $m$ copies of $n$ different card types are shuffled uniformly. In the former, they show that under an optimal strategy, uniformly in $n$, the expected number of correct guesses is at most $m+O(m^{3/4}\log m)$. In the latter, more practical strategies to real-life card games are discussed. These papers answered questions posed by Diaconis and Graham in 1981 \cite{Diaconis1981}. The bound in \cite{CardGuessingWithPartialFeedback} can be improved to $m+O(\sqrt{m})$ when $n=\Omega(\sqrt{m})$ as shown by Ziepi Nie in \cite{ZiepiNie}.

Lastly, the player may receive complete feedback, i.e. the card is turned face up. The $n$ cards of $m$ copies shuffled uniformly case has been studied by Ottolini and Steinerberger in \cite{OttoliniSteinerberger} and by He and Ottolini in \cite{HeOttolini}. Pengda Liu studied complete feedback guessing in the context of a riffle shuffle in \cite{Liu2020}. 

We apply some of the techniques discussed by Ciucu to shelf-shuffles in the no-feedback and complete-feedback cases. Fulman, Diaconis, and Holmes examined shelf-shuffles in detail in \cite[FDH]{FDH}. See their introduction for a thorough description of how casinos use shelf-shuffling machines, and for a description of shelf shuffles in general. They found a closed formula for the probability that a shelf shuffler with a given number of shelves and cards outputs a given permutation. In Section $\ref{Section2}$, we describe shelf shuffles, and we prove that the position matrix for a shuffle with one shelf is given explicitly as follows.

\begin{theorem}
\label{positionmatrix}
    Suppose a deck of $n$ cards is shuffled once in a shelf-shuffler with one shelf. Let $p_{i,j}$ be the probability that card $i$ lands in position $j$. Then, we have 
    \[p_{i,j}=\frac{\binom{i-1}{j-1}+\binom{i-1}{n-j}}{2^i},\;\;\;1\leq i\leq n,\;\;\;1\leq j\leq n\]
with the convention that $\binom{0}{0}=1$ and $\binom{0}{k}=0$ for all positive integers $k$. In particular, $p_{i,j}=0$ if and only if $i+1\leq j\leq n-i$, and $p_{i,n-j+1}=p_{i,j}$ for all $i,j$.  
\end{theorem}

The one-shelf case is a simple shuffle to describe, but the computations associated with it are difficult. We prove an optimal guessing strategy without feedback for a few positions. By optimal, we mean that the guessing strategy should maximize the expected number of correct guesses. In \cite{Ciucu}, Ciucu discusses that, with no feedback, a best guessing strategy is given by taking the argmax of each column in the position matrix. Then the player guesses the resulting ordered sequence of cards. 

\begin{theorem}
\label{no feedback guessing strategy}
    Suppose a deck of $n$ cards is given a shelf shuffle with a single shelf. The player should follow the following (incomplete) guessing strategy. Card $1$ should be guessed in positions $1$, and $n$ and cards $2$ or $3$ should be guessed in positions $2$ and $n-1$. If $n$ is even, card $n-1$ or card $n$ should be guessed in positions $\frac{n}{2}$ and $\frac{n}{2}+1$. If $n$ is odd, card $n-1$ or card $n$ should be guessed in position $\frac{n+1}{2}$. 
\end{theorem}

We conjecture an optimal no-feedback guessing strategy for the remaining cases. Roughly speaking, the player should guess card $2i-1$ in position $i$ until the middle of the deck is reached. Then, the player should guess his first $n/2$ guesses in reverse. For example, in an $8$ card deck, the player would guess cards $1,3,5,7,7,5,3,1$. The full strategy can be found in the appendix $\ref{Appendix}$.

Using the method of indicators, one can show that the expected number of correct guesses, without feedback, is the sum of the maxima of each column of the position matrix. We prove the following asymptotic as the deck size $n$ tends to infinity for the number of correct guesses using the conjectured strategy.

\begin{theorem} 
\label{No-feedback Asymptotic}
Suppose a deck of $n$ cards is given a shelf shuffle with a single shelf. Let $E(n)$ be the expected number of correct guesses in a deck of $n$ cards using the conjectured best strategy with no feedback. We have $E(n)\sim \sqrt{\frac{2}{\pi}}\sqrt{n-\sqrt{n}}+2\Phi(1)-1+C_n$ as $n\to\infty$, where $C_n\lesssim e^{-2}\cdot (2\pi)^{-1/2}\cdot (n- \sqrt{n})^{1/2}$. Here, $\Phi$ is the c.d.f. of an $\mathcal{N}(0,1)$ random variable.
\end{theorem}

While the no-feedback case is difficult to work with, the full feedback case has much nicer results. Liu discussed a complete-feedback guessing strategy in \cite{Liu2020} for a one-time riffle shuffle. We draw on his method for inspiration. Aldous and Diaconis applied stopping times to card shuffling in \cite{ShufflingAndStoppingTimes} from a group-theoretic perspective. To understand the shelf-shuffle case, we analyze the first descent in the permutation of the cards and, roughly speaking, associate a stopping time to it. Our stopping time is defined purely stochastically. Using this, we prove a unique optimal guessing strategy, with complete feedback. The process is defined in Section $3$, and we can also state it in the following words.

\begin{theorem}
\label{complete feedback optimal strategy}
    Suppose a deck of $n$ cards is given a shelf shuffle with a single shelf, and the player is given complete feedback. The player should guess card $1$ in position $1$, and the card immediately following the previously shown card in all other positions. Once one of card $n-1$ or $n$ is first shown, the player should guess the remaining cards in descending order, and is guaranteed to get these correct.
\end{theorem}
    For example, if $n=20$, say, and the first cards that come out are, in order, $1,\,2,\,5,\,10,\,11,\,19$, the player will have guessed cards $1,2,3,6,11,12$ and have gotten cards $1,\,2,\,11$ correct. After this, the player will guess cards $20,18,17,16,15,14,13,12,9,8,7,6,4,3$ and get them, almost surely, all correct. 

This strategy agrees with the strategy conjectured in \cite[FDH]{FDH} for one shelf. The authors ran Monte-Carlo simulations, and the number of correct guesses in a $52$-card deck using this strategy averaged $39$. We show that the position of the first descent is a shifted binomial random variable and prove the following expected value.

\begin{theorem}
\label{complete feedback expected reward}
    Suppose a deck of $n$ cards is given a shelf shuffle with a single shelf. Using the full-feedback optimal strategy defined in Theorem $\ref{complete feedback optimal strategy}$, the expected number of correct guesses is $3n/4$.
\end{theorem}

Finally, we include a conjecture which generalizes the complete-feedback expected reward and strategy to the $m$-shelf case. Our proofs require only some intermediate Markov chain techniques and should be approachable.

\section*{Acknowledgements}

We would like to thank Jason Fulman for kindly suggesting this problem and for his helpful comments with regards to this paper. In addition, thanks to Persi Diaconis for his references and feedback.
\section{The Position Matrix}
\label{Section2}
We establish the position matrix for a one-time shelf-shuffle with a single shelf. Before presenting our argument, we develop the modeling for shelf-shuffles and motivate our work. Shelf-shuffles are used to model casino card-shuffling machines. Several equivalent definitions of a shelf-shuffle are presented in \cite[Section 3]{FDH}. We will work exclusively with the following characterizations. 
\begin{definition}[$m$-shelf shuffles]
    A deck of cards is initially labeled $1,2,\ldots, n$. A shuffler has $m$ shelves. Cards are drawn from the bottom of the deck and placed on a shelf uniformly at random. Once a shelf is chosen, the card is placed at the top of the existing pile or at the bottom with probability $1/2$.  
\end{definition}

Setting $m=1$ allows us to obtain a fairly simple description of the shuffle which is the main interest of this paper.

\begin{definition}[Single-shelf shuffles] A deck of cards is initially labeled $1,2,\ldots, n$. Cards are drawn from the bottom of the deck and placed in a pile. When a card is drawn, it is placed at the top or at the bottom of the pile with probability $1/2$.
\end{definition}
    The shuffle in Definition 2.2 can also be characterized as a ``random" Monge shuffle.
\begin{definition}[Monge shuffles] 
     A Monge shuffle is a shuffle where cards from the top of the deck in one's left hand are drawn from the top. The cards are alternatively placed in the bottom and the top of the deck in one's right hand. 
\end{definition}
See Kraitchik \cite{Kraitchik} for a description and some results on Monge shuffles.

In the single-shelf case, the decision to place the card at the bottom or at the top of the deck in one's right hand is random. The random choice motivates the alternative name. Cards are also drawn from the bottom, not the top. 

This allows for an efficient combinatorial argument for the position matrix. Since we draw from the bottom, we have $n-i$ cards in the pile immediately before selecting the $i$th card. The $i$th card is placed at the bottom (resp. top) of the pile with probability $1/2$. If the card is placed at the top, select $j-1$ cards from the remaining $i-1$ cards and place them at the top. If the card is placed at the bottom, select $n-j$ cards from the remaining $i-1$ cards and place them at the bottom. These events are disjoint, so we may use the law of conditional probability and Theorem $\ref{positionmatrix}$ follows. 

We examine a few properties of $M$ that will allow for some detailed analysis in future sections.

\begin{proposition}
\label{doubly stochastic}
    Let $M$ be the position matrix as in Theorem $\ref{positionmatrix}$, i.e. set $M_{ij}=p_{i,j}$. We have that $M$ is doubly stochastic.
\end{proposition}
\begin{proof}
    First, we compute the row sums. We have 
    \[\sum_{j=1}^n M_{ij}=\sum_{j=1}^n\frac{\binom{i-1}{j-1}+\binom{i-1}{n-j}}{2^i}=\frac{1}{2}+\sum_{j=1}^n\frac{\binom{i-1}{n-j}}{2^i}\]
    where we used the binomial theorem and that $i-1<n$. Now observe that, re-indexing, 
    \[\sum_{j=1}^n\frac{\binom{i-1}{n-j}}{2^i}=\sum_{j=n-i+1}^n\frac{\binom{i-1}{n-j}}{2^i}=\sum_{k=0}^{i-1}\frac{\binom{i-1}{k}}{2^i}=\frac{1}{2}\]
    and this proves that the rows are stochastic.

    As for the columns, we note that we start and end with $n$ cards, so all positions are occupied with probability $1$.
\end{proof}

    Using facts about Markov chains, we obtain a few more facts about $M$.
    
\begin{corollary}
\label{stationary distribution}
    The stationary distribution of $M$ is uniform. 
    \begin{proof}
        We note that $M$ is irreducible and aperiodic. Therefore, the stationary distribution of $M$ is unique. A standard result from doubly stochastic Markov chains, see \cite[LPW]{LPW}, implies that the stationary distribution of $M$ is uniform.
    \end{proof}
\end{corollary}

\begin{corollary}
    $M$ has an eigenvalue $1$ with multiplicity $1$. $M$ has an eigenvalue $0$ with multiplicity $n/2$ if $n$ is even, and multiplicity $(n-1)/2$ if $n$ is odd.
    \begin{proof}
        By the Perron-Frobenius theorem, $M$ has an eigenvalue $1$ with multiplicity $1$. We also note that by Corollary $\ref{stationary distribution}$, one can take $\vec{v}=[1,1,...,1]^T$ as the corresponding eigenvector. Note that $M_{ij}=M_{i,n-j+1}$ and $M_{ij}=0$ if and only if $i+1\leq j\leq n-i$, so the rank of $M$ is $n/2$ if $n$ is even and $(n-1)/2$ if $n$ is odd. This shows that $0$ is an eigenvalue with the desired multiplicity.
    \end{proof}
\end{corollary}
We conjecture the form of an eigenvector corresponding to $1/4$. 
\begin{conjecture}
    Let $n\geq 3$. $M$ has an eigenvalue-eigenvector pair $(\frac{1}{4},\vec{v})$ with multiplicity at least $1$ and the $k$th entry of $\vec{v}$ given by
        \[v_k=n^2-3nk+\frac{3k(k+1)}{2}-1\]
\end{conjecture}

To construct $\vec{v}$, we ran Matlab code which numerically found the eigenvector. Then we properly normalized it and fit the entries to quadratic polynomials in $n$. Then, we recognized the coefficient of $n$ and the constant term as sequences in $k$. Numerical evidence strongly suggests the following form for the nonzero eigenvalues of $M$.

\begin{conjecture}
    The nonzero eigenvalues of $M$ are $1/2^{2i}$, $i=0,1,\ldots,\lceil(n-1)/2\rceil$.
\end{conjecture}

Unfortunately, our method described above is inefficient for finding further eigenvectors. The entries are monic degree $i+1$ polynomials in $n$. We are working on guessing a general formula by writing the entries as a sum of signed binomial coefficients in the style of Ciucu's work. Once we find the eigenvector-eigenvalue pairs explicitly, we may be able to compute higher powers of $M$. This will help to determine guessing strategies for a larger number of shuffles. For a larger number of shelves, we have an equivalence from \cite[FDH]{FDH} which provides an algorithmic way to compute transition probabilities.

\begin{proposition}
\label{multishelfcorrespondence}
    Let $m\geq 1$ and let $S_{2^{m-1}}$ be the transition matrix for a $2^{m-1}$-shelf shuffle. Let $M$ be the transition matrix for a single-shelf shuffle. Then, we have $S_{2^{m-1}}=M^m$.
    \begin{proof}
        By \cite[Corollary 4.2]{FDH}, a $2^{m-1}$-shelf shuffle is equivalent to a single-shelf shuffle done $m$ times.
    \end{proof}
\end{proposition}

This shows that we may relate single-shelf shuffling multiple times with adding an appropriate number of shelves. We remark that the transition matrix for an arbitrary number of shelves is an open problem. The main consequence of Proposition $\ref{multishelfcorrespondence}$ concerns the separation distance. Recall that the separation distance of a probability measure $P$ on a finite state space $\Omega$ is defined as
\[\text{sep}(P)=\max_{\omega\in\Omega}\left(1-\frac{P(\omega)}{U(\omega)}\right)\]
Here, $U(\cdot)$ is the uniform distribution on $\Omega$. Let $P$ be the measure on the symmetric group $S_n$ induced by a single-shelf shuffle. Additionally, let $U$ be the uniform measure on $S_n$, that is $U(\omega)=(n!)^{-1}$ for all $\omega\in S_n$. Set $P^s=P\ast P\ast \cdots\ast P$, the convolution of $P$ with itself $s$ times as in \cite{GroupRepsinProb}. Then $P^s$ is the measure induced by single-shelf shuffling $s$ times. By Proposition $\ref{multishelfcorrespondence}$, $P^s$ is equal to the measure induced by doing a $2^{s-1}$-shelf shuffle once. We use this to give an asymptotic for the separation distance.
\begin{proposition}
    Consider a single-shelf shuffle with $n$ cards done $s$ times. If
    \[s=d+\frac{3}{2}\log_2 n,\;\;\;\;\;d\in\mathbb{R}\;\;\text{fixed}\]
    then as $n\to\infty$, we have 
    \[\text{sep}(P^s)\sim 1-\exp\left(-\frac{1}{6\cdot 4^d}\right)\]
    \begin{proof}
        We may replace $s$ shelves with $2^{s-1}$ shuffles in \cite[Theorem 3.4]{FDH} by Proposition $\ref{multishelfcorrespondence}$. Then, for some $c>0$, if $2^{s-1}=cn^{3/2}$ and $n\to\infty$, the result holds with $d=1+\log_2 c$. The above is equivalent to $s=d+\frac{3}{2}\log_2 n$.
    \end{proof}
\end{proposition}
\section{No-Feedback Guessing}
\label{Section 3}
We provide a partial no-feedback guessing strategy for a one-time shelf shuffle. This section develops the theory needed to prove Theorem $\ref{no feedback guessing strategy}$. It also provides some preliminary work towards proving the conjectured strategy in the appendix. We present the proof of Theorem $\ref{no feedback guessing strategy}$ as a sequence of propositions. Before this, we build some of the intuition for a no-feedback guessing strategy by precisely defining the notion. Let
\[m_j^*=\max_{1\leq i\leq n}p_{i,j}\]
be the maximum entry of the $j$th column of $M$. The goal is to maximize the expected number of correct guesses with no feedback, denoted by $E_G$. If $G=\{g_j\}_{j=1}^n$ are the cards that are guessed in position $j$ and $r_j$ is the card that appears in position $j$, the method of indicators gives
\[E_G(n)=\sum_{j=1}^n p_{g_j,j}\]
which is clearly maximized by choosing $G^*=\{g_j\;|\;p_{g_j,j}=m_j^*\}$. That is, at each position $j$, we guess a card that has the highest probability of going to $j$. Note that, in general, $G^*$ is not unique, but $E_{G^*}(n)$ is independent of the best strategy $G^*$ picked. Later, we will suppress the $G^*$ notation and assume that a best strategy is always being used. Let
\[I_j^*=\{i\in [n]:\,\,p_{i,j}=m_j^*\}\]
be the set of possible best guesses at position $j$. Note that $I_j^*$ should also be indexed by $n$, but typically we will avoid this when $n$ is understood. Note that we may always construct a $G^*$ by selecting one element of each $I_j^*$ for $1\leq j\leq n$. 

From here on, we work with the one-time single-shelf shuffle from Section $\ref{Section2}$. Our first result shows that we only need to define a best guessing strategy for the first half of the deck.
\begin{proposition}
\label{guessing strategy symmetry}
    We have $I_j^*=I_{n-j+1}^*$ for all $1\leq j\leq n$, that is, the best guesses in position $j$ and in position $n-j+1$ are the same.
    \begin{proof}
        From the vertical symmetry axis of $M$ described in Theorem $\ref{positionmatrix}$, we have $p_{i,j}=p_{i,n-j+1}$, so the $j$th and $n-j+1$st columns are equal. Therefore, their argmaxes are equal.
    \end{proof}
\end{proposition}

The best guess in the first (and therefore last) positions is obvious.

\begin{proposition}
\label{first card}
    Let $n\geq 3$. Then, $I_1^*=I_n^*=\{1\}$, i.e. the player should guess card $1$ in positions $1$ and $n$.
    \begin{proof}
        Note that for a two-card deck, $p_{i,j}=1/2$ for all $i,j$, so any strategy is best. If $n\geq 3$, we notice that $p_{1,1}=p_{1,n}=1/2$ by direct computation or because card $1$ is drawn last, so it must go to the top or the bottom of the shuffled deck. Using Proposition $\ref{doubly stochastic}$, we have that the column sums of $M$ are $1$. Furthermore, one can verify that there exists an $i\neq 1$ such that $p_{i,1}>0$. Therefore $p_{i,n}>0$ for that same $i$ by Proposition $\ref{guessing strategy symmetry}$. The pigeonhole principle then implies the result.
    \end{proof}
\end{proposition}
For the second card, our argument requires a bit more work. We first prove a lemma, showing that the last two rows of $M$ are equal.
\begin{lemma}
\label{last two rows are equal}
    We have $p_{n-1,j}=p_{n,j}$ for all $1\leq j\leq n$.
    \begin{proof}
        We present a combinatorial argument. Card $n$ is drawn first and placed by itself. Now card $n-1$ goes on top, respectively on the bottom, of card $n$ with probability $1/2$. Hence, cards $n$ and $n-1$ move as a block in the shuffle with both orderings equally likely. Therefore, given a position, they are equally likely to end up there. One can also argue by computing the transition probabilities and using Pascal's identity and symmetry.
    \end{proof}
\end{lemma}
\begin{proposition}
    Let $n\geq 5$. Then, $I_2^*=I_{n-1}^*=\{2,3\}$, i.e. the player should guess card $2$ or card $3$ in positions $2$ and $n-1$.
    \begin{proof}
        Note that the $n=2$ case was covered in the proof of Proposition $\ref{first card}$. One verifies that when $n=3$, $p_{2,2}=p_{3,2}=1/2$. When $n=4$, we have $p_{3,2}=p_{4,2}=p_{3,3}=p_{4,3}=3/8$, and Proposition $\ref{doubly stochastic}$ implies that the player should guess card $3$ or $4$ in position $2$ and card $3$ or $4$ in position $3$.
        We will show that $p_{2,2}$ and $p_{3,2}$ are the unique maxima of column $2$. The column $n-1$ case will follow by Proposition $\ref{guessing strategy symmetry}$. We have, for any $i\leq n$, $p_{i,2}=\frac{i-1+\binom{i-1}{n-2}}{2^i}$. The binomial coefficient vanishes when $i-1<n-2$, i.e. when $i\leq n-2$. For the other two cases, we have $p_{n-1,2}=\frac{n-1}{2^{n-1}}=p_{n,2}$ using Lemma $\ref{last two rows are equal}$. So the problem has been reduced to finding
        \begin{equation}\label{max2}
        \max_{1\leq i\leq n-2}\frac{i-1}{2^i}
        \end{equation}
        and comparing this value with $p_{n,2}=\frac{n}{2^{n}}$. Some calculus shows that the maximum in equation $\ref{max2}$ is $1/4$ and is achieved only at $i=2$ and $i=3$. To compare this with $p_{n,2}$, we note that by the same calculus used above, for $n\geq 2$, the sequence $p_{n,2}=\frac{n}{2^n}$ is strictly decreasing with $p_{4,2}=1/4$. Hence, $p_{n,2}<1/4$ for all $n\geq 5$ and thus the maxima are unique.
    \end{proof}
\end{proposition}
Our last concrete result on the no-feedback strategy concerns the middle of the deck.
\begin{proposition}
\label{middle terms strategy}    
Let $n$ be even. Then, $I_{\frac{n}{2}}^*=I_{\frac{n}{2}+1}^*=\{n-1,n\}$, i.e. the player should guess card $n-1$ or card $n$ in positions $\frac{n}{2}$ and $\frac{n}{2}+1$. If $n$ is odd, card $n-1$ or card $n$ should be guessed in position $\frac{n+1}{2}$.
    \begin{proof}
        By Proposition $\ref{guessing strategy symmetry}$ and Lemma $\ref{last two rows are equal}$, it suffices to show that $n\in I_{\frac{n}{2}}^*$ and $k\notin I_{\frac{n}{2}}^*$ for $k\leq n-2$. We compute 
        \[p_{i,\frac{n}{2}}=\frac{\binom{i-1}{\frac{n}{2}-1}+\binom{i-1}{\frac{n}{2}}}{2^i}=\frac{\binom{i}{\frac{n}{2}}}{2^i}\]
        where the second equality is a consequence of Pascal's identity. Consider the sequence 
        \[a_i=\frac{\binom{i}{\frac{n}{2}}}{2^i}\]
        We note that $a_i=0$ for all $i\leq \frac{n}{2}-1$ and $a_i>0$ for all $i\geq \frac{n}{2}$. Now we examine the ratio of consecutive terms. Canceling factorials, we obtain
        \[R_i:=\frac{a_{i+1}}{a_i}=\frac{i+1}{2(i+1-\frac{n}{2})}\]
        Immediately, we have $R_{n-1}=1$ and $R_{i}<1$ for all $i\leq n-2$. This shows that $a_i$ is strictly increasing for all $\frac{n}{2}\leq i\leq n-1$ with $a_{n-1}=a_{n}$. This implies $a_{n-1}$ and $a_n$ are the unique maxima. But $a_n=p_{n,\frac{n}{2}}$, and the result follows. The construction for odd $n$ is nearly identical.
    \end{proof}
\end{proposition}

This completes the proof of Theorem $\ref{no feedback guessing strategy}$. We discuss a few techniques that we believe may be useful to prove the strategy conjectured in the appendix. Generally, the $\binom{i-1}{j-1}$ coefficient is easy to understand. In particular, we have the following general result to maximize weighted binomial coefficients of this form.

\begin{lemma}
\label{maximizing first binomial coefficient}
    We have
    \begin{equation}\label{maximum of first binomial coefficient}\max_{i\geq 1}\frac{\binom{i-1}{j-1}}{2^{i}}=\frac{\binom{2j-2}{j-1}}{2^{2j-1}}\end{equation}
    and this maximum is achieved if and only if $i=2j-1$ or $i=2j-2$.
    \begin{proof}
        Arguing as in the proof of Proposition $\ref{middle terms strategy}$, we let 
            \[a_i=\frac{\binom{i-1}{j-1}}{2^i}\;\;\;\text{and}\;\;\;R_i=\frac{a_{i+1}}{a_i}=\frac{i}{2(i-j+1)}\]
        We observe that $R_i>1$ when $i<2j-2$, $R_i=1$ when $i=2j-2$, and $R_i<1$ when $i>2j-2$. A similar argument as in Proposition $\ref{middle terms strategy}$ implies Equation $\ref{maximum of first binomial coefficient}$.
    \end{proof}
\end{lemma}
We notice that $\binom{i-1}{n-j}=0$ when $i\leq n-j$. This leads to a clue about the maximization for certain subsets of columns.
\begin{corollary}
    We have
    \begin{equation}\max_{1\leq i\leq n-j}p_{i,j}=\frac{\binom{2j-2}{j-1}}{2^{2j-1}}\end{equation}
    and this maximum is achieved if and only if $i=2j-1$ or $i=2j-2$.
\end{corollary}

Note that this is useful for the first (and last) third of the deck, as $2j-1\leq n-j$ if and only if $j\leq\frac{n+1}{3}$. This allows us to maximize the entries of each column above the antidiagonal. By Proposition $\ref{guessing strategy symmetry}$, this is the same as maximizing above the diagonal for the second half of the deck. Below the antidiagonal, resp. diagonal, the contribution of the $\binom{i-1}{n-j}$, resp. $\binom{i-1}{j-1}$, term cannot be ignored. Extending the arguments from Lemma $\ref{maximizing first binomial coefficient}$ may allow one to prove the full strategy in the appendix. 
    
\section{An Asymptotic for the No-Feedback Reward}
\label{Section 4}
Working with the setup from Section $\ref{Section2}$, we develop a moderately precise asymptotic for the expected number of correct guesses without feedback. In this section, we assume that the transition probabilities are as in Theorem $\ref{positionmatrix}$. We also will use the numerically-supported best strategy from the appendix. The construction in Section $\ref{Section 3}$ guarantees that we are free to choose any of the optimal strategies to compute the expectation. By carefully choosing such a strategy and applying limit theorems and some analysis, we prove Theorem $\ref{No-feedback Asymptotic}$ in three parts.

We start with a few preliminaries. Without loss of generality, let $n$ be even. Set
\[m=\left\lceil\frac{\sqrt{n-1}}{2}-1\right\rceil\]
and adopt the following (optimal, as in the appendix $\ref{Appendix}$) guessing strategy for the first half of the deck. Guess card $1$ in position $1$ and then guess card $2j-1$ in position $j$ for $2\leq j\leq \frac{n}{2}-m-1$. Then, guess card $n$ in position $j$ for $j\leq \frac{n}{2}-m-1$. By Proposition $\ref{guessing strategy symmetry}$, the expected reward for the second half of the deck is equal to the expected reward for the first half. Therefore, suppressing the $G^*$ notation as in Section $\ref{Section 3}$, we have
\[E(n)=\sum_{j=1}^n m_j^*=2\sum_{j=1}^{\frac{n}{2}}m_j^*\]
\[=2\sum_{j=1}^{\frac{n}{2}-m-1}\frac{\binom{2j-2}{j-1}+\binom{2j-2}{n-j}}{2^{2j-1}}+2\sum_{j=\frac{n}{2}-m}^{\frac{n}{2}}\frac{\binom{n-1}{j-1}+\binom{n-1}{n-j}}{2^n}\]
Now, splitting $E(n)$ into pieces, denote
\[S_1=2\sum_{j=1}^{\frac{n}{2}-m-1}\frac{\binom{2j-2}{j-1}}{2^{2j-1}}\;\;\;\;\;\;\; S_2=2\sum_{j=1}^{\frac{n}{2}-m-1}\frac{\binom{2j-2}{n-j}}{2^{2j-1}}\]
and
\[S_3=\sum_{j=\frac{n}{2}-m}^{\frac{n}{2}}\frac{\binom{n-1}{j-1}+\binom{n-1}{n-j}}{2^n}\]
We analyze each of these sums.
\begin{proposition}
    We have $S_1\sim\sqrt{\frac{2}{\pi}}(n-\sqrt{n})^{1/2}$
    \begin{proof}
        We may compute $S_1$ explicitly. We have, via a combinatorial identity,
        \[S_1=\frac{n-2m-3}{2^{n-2m-4}}\binom{n-2m-4}{\frac{n}{2}-m-2}\]
        By a special case of Stirling's estimate, we have
        \[\binom{n}{n/2}\sim 2^n(\pi n/2)^{-1/2}\] 
        This gives
        \[S_1\sim \frac{(n-2m-3)(2^{n-2m-4})}{(2^{n-2m-4})\sqrt{\pi(\frac{n}{2}-m-2)}}=\frac{(n-2m-3)}{\sqrt{\pi(\frac{n}{2}-m-2)}}\]
        \[=\sqrt{\frac{2}{\pi}}\left(\frac{n-2m-3}{\sqrt{n-2m-4}}\right)\]
        \[=\sqrt{\frac{2}{\pi}}\left(\frac{n-\left\lceil\sqrt{n-1}\right\rceil-2}{\sqrt{n-\left\lceil\sqrt{n-1}\right\rceil-3}}\right)\]
        \[\sim\sqrt{\frac{2}{\pi}}(n-\sqrt{n})^{1/2}\]
    \end{proof}
\end{proposition}
To analyze $S_2$ and $S_3$, we need the following lemmas. This lemma will be useful for both sums.
\begin{lemma}[de Moivre-Laplace]
\label{de Moivre-Laplace}
Suppose 
\[\frac{k_n-np}{\sqrt{np(1-p)}}\to x\in\mathbb{R}\]
or, equivalently, $k_n=np+O(\sqrt{n})$.
Then,
\[\binom{n}{k_n}p^{k_n}(1-p)^{n-k_n}\sim\frac{1}{\sqrt{2\pi np(1-p)}}e^{-\frac{x^2}{2}}\]
\end{lemma}

The next lemma will be applicable to our analysis of $S_2$.

\begin{lemma}
\label{upper bound for s2 summand}
    We have
    \[\max_{1\leq j\leq\frac{n}{2}-m-1}\frac{\binom{2j-2}{n-j}}{2^{2j-2}}=\frac{\binom{n-2m-4}{\frac{n}{2}+m+1}}{2^{n-2m-4}}\]
    \begin{proof}
        Denote 
        \[a_j=\frac{\binom{2j-2}{n-j}}{2^{2j-2}}\]
        First notice that $a_j>0$ if and only if $j\geq\frac{n+2}{3}$. Fix $n\geq 5$. We will apply the ratioing method from the proof of Proposition $\ref{middle terms strategy}$. We have
        \[\frac{a_{j+1}}{a_j}=\frac{(2j)(2j-1)(n-j)}{4(3j-n+1)(3j-n)(3j-n-1)}:=f_n(j)\]
        We sketch the rest of the proof. One can show by differentiation that when $n\geq 5$ is fixed, $f$ is decreasing on $j\geq\frac{n+2}{3}$, which is precisely in the nonzero terms of $S_2$. By direct computation, we have $f(\frac{n}{2})<1$ and $f(\frac{n}{2}-1)>1$. Since $f$ is decreasing, this implies 
        \begin{equation}
        \label{s2 ratio equation}
        f_n(j)>1\;\;\;\; \text{for all}\;\;\;\; \frac{n+2}{3}\leq j\leq\frac{n}{2}-1. 
        \end{equation}
        In particular, since $\frac{n}{2}-m-1<\frac{n}{2}-1$ when $n\geq 5$, equation $\ref{s2 ratio equation}$ holds for all $\frac{n+2}{3}\leq j\leq\frac{n}{2}-m-1$. We deduce that $a_j$ is maximized at $j=\frac{n}{2}-m-1$.
    \end{proof}
\end{lemma}

The preceding maximization result allows us to give a crude asymptotic bound for $S_2$ which is sufficient for our analysis.

\begin{proposition}
    We have $S_2\lesssim e^{-2}\cdot (2\pi)^{-1/2}\cdot (n- \sqrt{n})^{1/2}$.
    \begin{proof}
        From the uniform upper bound in Lemma $\ref{upper bound for s2 summand}$, we have
        \[S_2\leq\frac{(\frac{n}{2}-m-1)\binom{n-2m-4}{\frac{n}{2}+m+1}}{2^{n-2m-4}}\]
        We check that
        \[\frac{\frac{n}{2}+m+1-\frac{n-2m-4}{2}}{\sqrt{(n-2m-4)(\frac{1}{4}})}\]
        \[\frac{4m+6}{\sqrt{(n-2m-4)}}=\frac{2\,O(\sqrt{n})}{O(\sqrt{n})}\to 2\]
        so the conditions of Lemma $\ref{de Moivre-Laplace}$ are satisfied. This means
        \[\frac{(\frac{n}{2}-m-1)\binom{n-2m-4}{\frac{n}{2}+m+1}}{2^{n-2m-4}}\sim\frac{(\frac{n}{2}-m-1) (e^{-2})}{\sqrt{2\pi(n-2m-4)(\frac{1}{4}})}\]
        \[\sim e^{-2}\cdot (2\pi)^{-1/2}\cdot (n- \sqrt{n})^{1/2}\]
        which establishes asymptotic control over $S_2$.
    \end{proof}
\end{proposition}

The last piece is $S_3$, which has an explicit, finite asymptotic. The proof relies on Lemma $\ref{de Moivre-Laplace}$.

\begin{proposition}
    We have $S_3\sim 2\Phi(1)-1\approx0.68$, where $\Phi$ is the cumulative distribution function of an $\mathcal{N}(0,1)$ random variable.
    \begin{proof}
        Using symmetry and re-indexing, we may rewrite
            \[S_3=2\sum_{i=\frac{n}{2}-m}^{\frac{n}{2}}\frac{\binom{n-1}{i-1}}{2^{n-1}}=2\sum_{k=\frac{n}{2}-m-1}^{\frac{n}{2}-1}\frac{\binom{n-1}{k}}{2^{n-1}}\]
        We check that
        \[\frac{n}{2}-m-1=\frac{n-1}{2}-\left\lceil\frac{\sqrt{n-1}}{2}-\frac{1}{2}\right\rceil-\frac{1}{2}=\frac{n-1}{2}+O(\sqrt{n})\]
        and this is clearly the farthest term from the central binomial coefficient. Lemma $\ref{de Moivre-Laplace}$ implies that the central limit theorem holds. By a simple algebraic manipulation of the endpoints of the sum, we obtain
        \[S_3\sim 2\,\Phi\left(\frac{2m+1}{\sqrt{n-1}}\right)-2\,\Phi\left(\frac{1}{\sqrt{n-1}}\right)\]
        \[=2\,\Phi\left(\frac{\lceil\sqrt{n-1}\rceil}{\sqrt{n-1}}\right)-2\,\Phi\left(\frac{1}{\sqrt{n-1}}\right)\]
        \[\to 2\,\Phi(1)-1\]
        by the continuity of $\Phi$.
    \end{proof}
\end{proposition}

Adding $S_1$, $S_2$, and $S_3$, we get the asymptotic in Theorem $\ref{No-feedback Asymptotic}$.
\begin{remark}
    We note that $\sqrt{\frac{2}{\pi}}(n-\sqrt{n})^{1/2}+0.68$ is an effective estimate for $E(n)$. Numerical evidence suggests that there is a tighter asymptotic bound for $S_2$. See the below table for some of our data, where the average number of correct guesses is the sum of the maximal probabilities.
\end{remark}

\begin{center}
\begin{tabular}{ | m{10em} | m{1cm} | m{1cm} | m{1cm} | m{1cm} |}
  \hline
  $n$ &  $10$ & $21$ & $33$ & $52$\\ 
  \hline
  Average number of correct guesses & $2.70$ & $3.88$ & $4.82$ & $6.00$\\
  \hline
  $\sqrt{\frac{2}{\pi}}(n-\sqrt{n})^{1/2}+0.68$ & $2.77$ & $3.91$ & $4.85$ & $6.02$\\
  \hline
  \end{tabular}
\end{center}

If the cards were shuffled perfectly (uniformly), one would expect to get exactly one card correct in a deck of $n$ cards. In a $52$-card deck, one expects to get $6$ cards right, indicating that the one-shelf shuffle is far from being perfect.

Although the no-feedback case is difficult to analyze, our asymptotic provides a precise way to estimate the expected reward for larger decks. We note that the rate of convergence seems quite fast and is worth studying.

\section{The Strategy with Complete Feedback}
\label{Section 5}
The complete-feedback case yields a more fruitful analysis. We develop an optimal complete-feedback guessing strategy for the one-time single-shelf shuffle described in Section $\ref{Section2}$. Our construction of this strategy allows for a simple computation of the expected reward. 

First, we give a general formulation of a guessing strategy with complete feedback. See \cite[1.4]{Liu2020} for a definition using maps of sequences. We define it in terms of the filtration generated by the cards that have already appeared, in order. This is natural because the information available to the player is increasing with respect to time. Immediately before card $j+1$ is guessed, the player knows which cards were in each of the $j$ previous positions. 

Let $r_j$ be the card appearing in position $j$ and let $\mathcal{F}_j$ be the $\sigma$-algebra generated by $\{r_1,r_2,...,r_j\}$ as an ordered set. Observe that $\mathcal{F}:=\{F_j\}_{j\geq 1}$ is a filtration. Set $\mathcal{F}_0=\varnothing$ and view $\{r\}:=\{r_1,r_2,\ldots,r_n\}$ as a process.
In general, we may define a guessing strategy $\{G\}:=\{g_1,g_2,\ldots,g_n\}$ as a process adjacent to $\{r\}$. As it is a strategy known to the player, $G$ should be $\mathcal{F}$-predictable. We demand that $G$ satisfies
\[g_{j+1}\in\{\text{arg}\max_{1\leq c\leq n}P(r_{j+1}=c\,|\,\mathcal{F}_j)\}\]
that is, $g_{j+1}$ maximizes the $\mathcal{F}_j$-conditional transitional probabilities for every $j$. For simplicity, we let $g_1$ follow the no-feedback guess. One verifies that choosing a $G$ according to the construction above defines an $\mathcal{F}$-predictable process. Note that $G$ may not be (conditionally) unique, but it must be predictable. That is, the player should know before the game is played which card he will guess in each possible scenario. For example, suppose that in position $4$, cards $7$ and $10$ appear with equal and maximal conditional probability. The player should know before the game is played that he will guess card $7$ if this scenenario comes up. 

Denote by $F^G(n)$ the expected reward with complete feedback. Constructing similarly to $E^G(n)$ from Section $\ref{Section 3}$, we have
\begin{equation}
\label{complete feedback construction}
F^G(n)=\sum_{j=1}^n P(g_j =r_j\,|\,\mathcal{F}_{j-1})
\end{equation}

Liu gave an equivalent definition with respect to the induced conditional distributions on elements of $S_n$ by shuffling algorithms in \cite[1.4]{Liu2020}. The proof of Liu's strategy for a one-time riffle shuffle relies on a recursive method. Once the first card is revealed, the remaining $n-1$ cards can be viewed as shuffled according to a conditional distribution. 

The one-time shelf shuffle yields a simpler result. The conditional distribution induced is simply a shifted version of the original one. More precisely, the information stored in $\mathcal{F}_{j-1}$ can be vast. Luckily, we only need $r_{j-1}$ (and $j-1$) to uniquely determine the conditional probabilities. This is a major difference from most other shuffles, including the one-time riffle-shuffle. We show this in Proposition $\ref{properties of r}$. 

It is possible to introduce a threshold after which the cards, $\{r\}$, become $\mathcal{F}$-predictable. Prior to this time, $\{r\}$ can be viewed as a (stopped) Markov chain. The authors in \cite[FDH]{FDH} alluded to this event as the first descent in the permutation of cards. We have an even more explicit characterization. Let
\[\tau^{n-1}=\inf_{j\geq 1}\{r_j=n-1\}\;\;\;\text{and}\;\;\;\tau^{n}=\inf_{j\geq 1}\{r_j=n\}\]
be the positions at which card $n-1$ and card $n$ occur, respectively. Note that the $\inf$ notation may be dropped as these cards both occur once. We use it as the reader may be familiar with it in the context of random walks. 

Using the notation $a\wedge b=\min(a,b)$, set $\tau=\tau^{n-1}\wedge\tau^{n}$. Observe that the event $\{\tau=m\}$ is the first time either card $n-1$ or $n$ is shown. According to the definition in \cite[PTE, 4.2]{Durrett}, $\tau$ is an $\mathcal{F}$-stopping time. Notice that $|\tau^n-\tau^{n-1}|=1$ a.s. using the proof of Lemma $\ref{last two rows are equal}$. 

To observe a few properties of the stopped process $r_{j\wedge\tau}$, we construct the following process that runs adjacently. Define the random sequence 
\[K_{j\wedge\tau}=([r_j]\,\backslash\,\{r_1,\,r_2,\,\ldots,\,r_j\})_{\text{rev}}=\{r_j-1,\,r_j-2,\,\ldots,r_{j-1}+1,\,r_{j-1}-1,\,\ldots,r_1+1,r_1-1,\,\ldots,\,1\}\]
For example, in a deck of size larger than $9$, if the first three cards, in order, are cards $2,\,4,\,7$, then $K_3=\{6,\,5,\,3,\,1\}$. We note that $r_{j\wedge\tau}$ is increasing and $K_{j\wedge\tau}$ is in descending order. 
\begin{lemma}
    We have $P(r_{j\wedge\tau}<r_{(j+1)\wedge\tau}\,|\,\tau>j)=1$ and, therefore, $K_{j\wedge\tau}$ is almost surely in descending order.
    \begin{proof}
        The first descent in the permutation of cards that appear is uniquely seen after card $n$ is shown. Prior to card $n$ being shown, the cards are in ascending order. After card $n$ is shown, the cards are in descending order. This is because card $n$ is the first card placed into the pile. Cards must be placed, from the bottom of the deck, on top or on the bottom of the pile. This preserves the relative order if placed on the top and reverses the relative order if placed on the bottom.
    \end{proof}
\end{lemma}

Equipped with this information, we discover that once card $n-1$ is seen, the remaining cards are in descending order a.s. We also see that, due to the cards being in ascending order prior to when card $n$ is seen, $r_{j\wedge\tau}$ is Markov. This allows us to compute the conditional transition probabilities explicitly, which have maxima $1/2$ and $1$ before and after $\tau$ is reached, respectively.

\begin{proposition}\label{properties of r}

    i.) The stopped process $r_{j\wedge\tau}$ is homogeneous and Markov. The transition probabilities of $r_j$ when both cards $n-1$ and $n$ have not been shown are 
\begin{equation}\label{conditionaltransition}
    P(r_{j+1}=i\,|\,r_j=k,\,\tau>j)=\frac{\binom{i-k-1}{0}+\binom{i-k-1}{n-k-1}}{2^{i-k}}=\frac{1_{\{k+1\leq i\leq n\}}+1_{\{i=n\}}}{2^{i-k}}
\end{equation}

ii.) $\tau-1$ has a $\text{Binomial}(n-2,\,1/2)$ distribution.

iii.) After the $n-1$st or $n$th card is shown, the cards are deterministic and their order is given uniquely by $K_{\tau}$, if $\tau=\tau^n$. If $\tau=\tau^{n-1}$, then $r_{\tau+1}=n$ a.s. and the remaining cards, in order, are given by $K_{\tau}$. That is,
\[P(r_{\tau+m}=i\,|\,\tau=\tau^{n})=1_{\{K_{\tau}^m=i\}}\]
and
\[P(r_{\tau+1}=n\,|\,\tau=\tau^{n-1})=1, \;\;\;\;P(r_{\tau+m+1}=i\,|\,\tau=\tau^{n-1})=1_{\{K_{\tau}^m=i\}}\]

\begin{proof}

    First we prove i.). Fix $j<\tau$ and $k$ such that $p_{k,j}>0$. Recall that cards go to the top and the bottom of the pile via a $\text{Bernoulli}(1/2)$ distribution. Let $N_k$ be the number of cards at the top up to and including when card $k$ is selected. Then, $N_k$ is $\text{Binomial}(k,\,1/2)$. Observe that $\{r_j=k\}=\{N_k=j\}$ as events (intersected with $\{j<\tau\}$). Also notice that the number of cards that have not gone to the top is $k-N_k$. So we are left with a shuffle of $n-(N_k+k-N_k)=n-k$ cards according to the distribution from Theorem $\ref{positionmatrix}$, with card $k+1$ at the top of the deck. This proves i.). 

    To prove ii.), we first notice that cards $n-1$ and $n$ always appear consecutively in the deck, with both orderings having probability $1/2$. Label this block of two cards $B$ and shuffle the remaining $n-2$ cards. Then $\{\tau=k\}$ if and only if $k-1$ cards have been selected, each with probability $1/2$, to go above $B$. Hence, 
    \[P(\tau=k)=\frac{\binom{n-2}{k-1}}{2^{k-1}}\]
    so that 
    \[P(\tau-1=k)=\frac{\binom{n-2}{k}}{2^k}\]
    showing ii.). 

    As for iii.), we note that after card $n$ is shown, the cards are almost surely in strictly decreasing order. This is because cards are drawn from the bottom of the deck. There is one unique arrangement of the remaining cards in decreasing order, namely, by removing the already selected cards from $[n]$ and reversing. This is exactly as in the construction of $K_{j\wedge \tau}$, proving iii.).
\end{proof}
\end{proposition}
We now prove Theorem $\ref{complete feedback optimal strategy}$ by defining a strategy $\mathbf{G}$ and showing it is optimal. 

\begin{proposition}\label{best strategy with feedback}
    Let $\mathbf{G}=\{g_j\}_{j=1}^n$ be the following process.
    
    i.) $g_1=1$ a.s.
    
    ii.) $g_{j+1}=r_j +1$, $\tau>j$. If $r_j=n-2$, then $g_{j+1}=n-1$ or $n$.
    
    iii.) If $\tau=\tau^{n-1}$, then $g_{\tau+1}=n$ a.s. and $g_{\tau+m+1}=K_{\tau}^m$ a.s. Otherwise, if $\tau=\tau^{n}$, then $g_{\tau+m}=K_{\tau}^m$ a.s.
    Then, $\mathbf{G}$ maximizes $F^G$ for every $n$. 
    \begin{proof}
        For the first card, we resort to the no-feedback case. This proves i.)
        By Proposition $\ref{properties of r}$ i.), we have 
        \begin{equation}
        \label{Complete Feedback Maxima}   P(r_{j+1}=r_j+1\,|\,r_j,\,\tau>j)=\frac{1}{2}
        \end{equation}
        for every $1\leq j\leq n-2$. Applying a pigeonhole argument as in Proposition $\ref{first card}$, we deduce ii.). Part iii.) follows from Proposition $\ref{properties of r}$ iii.), as the probabilities are all $1$ and thus must be the unique maxima.
    \end{proof}
    \end{proposition}

This completes the proof of Theorem $\ref{complete feedback optimal strategy}$. 

\begin{remark}
    The strategy in Theorem $\ref{complete feedback optimal strategy}$ was conjectured in \cite[FDH]{FDH}. We note that our strategy is equivalent because the first descent is observed a.s. immediately after card $n$ appears, and then the cards are in descending order.
\end{remark}

\section{Expected Reward With Complete Feedback}
\label{Section 6}
We prove that the expected reward using the optimal strategy $\mathbf{G}$ from Proposition $\ref{best strategy with feedback}$ for a one-time single-shelf shuffle is $3n/4$.  We can compute the expectation by conditioning on $\tau$. This is far less involved than summing through the $\mathcal{F}_{j-1}$-conditional expectations as in equation $\ref{complete feedback construction}$. To prove Theorem $\ref{complete feedback expected reward}$, we calculate $F^{G}(n)$ through this method. Our proof relies on two concepts. First, $\tau$ is a shifted binomial random variable by Proposition $\ref{properties of r}$ ii.). This allows us to use a variant of the law of iterated expectations. Second, the construction of $\mathbf{G}$ allows us to break the expectation into two parts. Recall the definition of the expected reward from equation $\ref{complete feedback construction}$. We compute
\[F^{\mathbf{G}}(n)=\sum_{j=1}^n E\left[\sum_{m=1}^n1_{\{\tau=m\}}1_{\{g_j=r_j\}}\,\bigg|\,\mathcal{F}_{j-1}\right]\]
\begin{equation}
\label{separating expectations}
=E\sum_{j=1}^{\tau}1_{\{g_j=r_j\}}+E\sum_{j=\tau+1}^n1_{\{g_j=r_j\}}
\end{equation}
where we used that $r_{j\wedge\tau}$ is Markov.
By inserting a conditional expectation into the first sum, we have 
\begin{equation}
\label{insert conditional expectation}
E\sum_{j=1}^{\tau}1_{\{g_j=r_j\}}=E\left[\sum_{j=1}^{\tau}E [1_{\{g_j=r_j\}}\,|\,\tau>j]\right]
\end{equation}
From equation $\ref{Complete Feedback Maxima}$ and Proposition $\ref{best strategy with feedback}$ i.), we have 
\[E[1_{\{g_j=r_j\}}\,|\,\tau>j]=E[1_{\{r_{j-1}+1=r_j\}}\,|\,\tau>j]=\frac{1}{2}\]
so that
\[E\sum_{j=1}^{\tau}1_{\{g_j=r_j\}}=\frac{E[\tau]}{2}=\frac{n}{4}\]
by Proposition $\ref{properties of r}$ ii.).

As for the second sum, notice that $r_{
\tau+m}$ is $\mathcal{F}_\tau$-predictable for every $1\leq m\leq n-\tau$. In particular, the cards are deterministic and are in descending order. Using the formal construction from Proposition $\ref{best strategy with feedback}$ iii.) and rewriting $j=\tau+m$ for some $m$, we have that $1_{\{g_j=r_j\}}=1$ a.s. conditional on $j>\tau$. Inserting a conditional expectation as in equation $\ref{insert conditional expectation}$ gives
\[E\sum_{j=\tau+1}^n1_{\{g_j=r_j\}}=E\sum_{j=\tau+1}^n E[1_{\{g_j=r_j\}}\,|\,\tau<j]\]
\[=E\sum_{j=\tau+1}^n1=E[n-\tau]=n-\frac{n}{2}=\frac{n}{2}\]
Putting both sums together, we have $F^{\mathbf{G}}(n)=3n/4$ as desired. The proof of Theorem $\ref{complete feedback expected reward}$ is complete.

\begin{remark}
    The expectation we find here agrees with the data found in \cite[FDH]{FDH}. They found that the expected number of correct guesses in a deck of $52$ cards was $39$.
\end{remark}

For the general case of an $m$-shelf shuffle, we have the following conjecture. Our conjecture is for the case when the ratio of cards to shelves is not too small. This is applicable in most practical cases, including in casinos with a $52$-card deck and ten shelves.

\begin{conjecture}
    Let $H_{2m}$ be the $2m$th harmonic number. Suppose a deck of $n$ cards is given an $m$-shelf shuffle, and $n/m$ is not too small. Then, the strategy, call it $\mathbf{G}$, discussed in \cite{FDH} can be shown to be optimal in a high-probability case. Let $F^{\mathbf{G}}(n,m)$ be the expected number of correct guesses. We have, when $n/m$ is not too small,
    \[F^{\mathbf{G}}(n,m)\approx\frac{n}{2m}H_{2m}=\frac{n}{2m}\sum_{k=1}^{2m}\frac{1}{k}\]
\end{conjecture}

We reach this conjecture with the following heuristic argument. After shuffling, suppose the deck is cut into $2m$ labeled piles, some possibly empty, with each pile representing the labeling each card was given as in the description in \cite[FDH]{FDH}. The average size of each pile is $n/2m$. If this quantity is greater than $2$, then the average pile should look like a monotone sequence to the player. One can outline an argument that the general strategy discussed in \cite[FDH]{FDH} is optimal when the piles are monotone. In the $k$th pile, the probability that the next available (monotone) card is in the next position is simply the probability that it was assigned to shelf $k$, conditional on it not being in the previous $k-1$ shelves, i.e. $(2m-k+1)^{-1}$. So, essentially, the distribution of correct guesses is akin to $n/2m$ copies of a $2m$-card deck each shuffled uniformly and separately. See Diaconis and Fulman \cite{MathofShufflingCards} for a description of this case. Summing over the $k$ piles and using the fact that the average size of each pile is $n/2m$, we obtain, after re-indexing, the conjecture. See the table of values below for some data from \cite[FDH]{FDH} along with our estimate. Note that the data from \cite[FDH]{FDH} is the sample expected value which corresponds to the definition which maximizes the conditional probabilities. In particular, it is potentially more accurate than data using the strategy $\mathbf{G}$ from their conjecture. Again, when $n/2m$ is too small, our estimate is not good.
\begin{center}
\begin{tabular}{ | m{10em} | m{1.2cm} | m{1.2cm} | m{1.2cm} | m{1.2cm} | m{1.2cm} | m{1.2cm} |}
  \hline
  $(n,m)$ &  $(52,1)$ & $(52,2)$ & $(52,4)$ & $(52,10)$ & $(52,20)$ & $(52,40)$\\ 
  \hline
  Average number of correct guesses & $39$ & $27$ & $17.6$ & $9.3$ & $6.2$ & $4.7$\\
  \hline
  $\frac{n}{2m}H_{2m}$ & $39$ & $27.08$ & $17.67$ & $9.35$ & $5.56$ & $3.23$\\
  \hline
  \end{tabular}
\end{center}
In our future work, we hope to make this argument precise.

\newpage
\section{Appendix}
\label{Appendix}
We give full conjectured optimal no-feedback guessing strategy. This strategy is based off of data obtained from generating the position matrices from Theorem $\ref{positionmatrix}$ for decks of various sizes $n$. The cases $j=1$, $j=2$, and $j$ the median card(s) of the deck, are proven in Section $\ref{Section 3}$ to be optimal. Let
\begin{equation}\label{alphacases}
\alpha_n=\begin{cases}
    n/3,& n\equiv0\mod{3}\\
    \lfloor n/3\rfloor, & n\equiv 1\mod{3}\\
    \lceil n/3\rceil, & n\equiv 2\mod{3}
\end{cases}
\end{equation}

The best strategy is fully characterized by the below cases for the first half of the deck. By Proposition $\ref{guessing strategy symmetry}$, this also gives the strategies for the second half of the deck, in reverse.

Suppose $n$ is odd. If there exists a nonnegative integer $m$ such that $n=(2m)^2+1$, i.e. $n\in S:=\{1,5,17,37,65,101,...\}$, then the best strategy is given below.
\begin{center}
\begin{tabular}{ | m{10em} | m{3cm} |} 
  \hline
  Position range & Best guess(es) at position $j$:  \\ 
  \hline
  $1\leq j\leq \alpha_n$ &  $2j-2,\,2j-1$\\ 
  \hline
  $\alpha_n+1\leq j\leq \frac{n+1}{2}-m-1$ & $2j-1$\\
  \hline
  $j=\frac{n+1}{2}-m$& $n-3,\,n-2,\,n-1,\,n$ \\ 
  \hline
  $\frac{n+1}{2}-m+1\leq j\leq \frac{n+1}{2}$& $n-1,\,n$\\
  \hline
\end{tabular}
\end{center}

Suppose $n$ is odd and $n\notin S$. Then, there exists a unique nonnegative integer $m$ such that $(2m)^2+1<n<(2(m+1))^2+1$. The best strategy is given below.

\begin{center}
\begin{tabular}{ | m{10em} | m{3cm} |} 
  \hline
  Position range & Best guess(es) at position $j$ \\ 
  \hline
  $1\leq j\leq \alpha_n$ &  $2j-2,\,2j-1$\\ 
  \hline
  $\alpha_n+1\leq j\leq \frac{n+1}{2}-m-1$ & $2j-1$\\
  \hline
  $\frac{n+1}{2}-m\leq j\leq \frac{n+1}{2}$& $n-1,\,n$\\
  \hline
\end{tabular}
\end{center}
Suppose $n$ is even. If there exists a nonnegative integer $m$ such that $n=(2m+1)^2+1$, i.e. $n\in S':=\{2,10,26,50,82,122,...\}$, then the best strategy is given below. 

\begin{center}
\begin{tabular}{ | m{10em} | m{3cm} |} 
  \hline
  Position range & Best guess(es) at position $j$ \\ 
  \hline
  $1\leq j\leq \alpha_n$ &  $2j-2,\,2j-1$\\ 
  \hline
  $\alpha_n+1\leq j\leq \frac{n}{2}-m-1$ & $2j-1$\\
  \hline
  $j=\frac{n}{2}-m$& $n-3,\,n-2,\,n-1,\,n$ \\ 
  \hline
  $\frac{n}{2}-m+1\leq j\leq \frac{n}{2}$& $n-1,\,n$\\
  \hline
\end{tabular}
\end{center}
Suppose $n$ is even and $n\notin S'$. Then, there exists a unique nonnegative integer $m$ such that $(2m+1)^2+1<n<(2(m+1)+1)^2+1$. The best strategy is given below.
\begin{center}
\begin{tabular}{ | m{10em} | m{3cm} |}
\hline
Position range & Best guess(es) at position $j$ \\ 
  \hline
  $1\leq j\leq \alpha_n$ &  $2j-2,\,2j-1$\\ 
  \hline
  $\alpha_n+1\leq j\leq \frac{n}{2}-m-1$ & $2j-1$\\
  \hline
  $\frac{n}{2}-m\leq j\leq \frac{n}{2}$& $n-1,\,n$\\
  \hline
  \end{tabular}
\end{center}

\newpage
\providecommand{\bysame}{\leavevmode\hbox to3em{\hrulefill}\thinspace}
\providecommand{\MR}{\relax\ifhmode\unskip\space\fi MR }
\providecommand{\MRhref}[2]{%
  \href{http://www.ams.org/mathscinet-getitem?mr=#1}{#2}
}
\providecommand{\href}[2]{#2}

Department of Mathematics, University of Southern California, Los Angeles, CA 90089-2532, USA

\textit{Email address:} \href{mailto:ajclay@usc.edu}{ajclay@usc.edu} 
\end{document}